\documentclass[12pt, a4paper]{article}

\usepackage[utf8]{inputenc}
\usepackage[T1]{fontenc}
\usepackage[english]{babel}

\usepackage[left=3cm, right=3cm, top=4cm]{geometry}

\usepackage[dvipsnames]{xcolor}
\usepackage{pgf,tikz}
\usetikzlibrary{arrows}

\usepackage[format=hang,font=small]{caption}
\usepackage{booktabs,bm,multirow,subfig}
\usepackage{enumitem,comment}
\usepackage{mathrsfs} 

\usepackage{amsmath,amsfonts,amssymb,amsthm}
\usepackage{dsfont,multicol}

\renewcommand{\theta}{\varphi}
\renewcommand{\theta}{\vartheta}
\renewcommand{\epsilon}{\varepsilon}

\newcommand{\R}{\mathbb{R}}
\newcommand{\Z}{\mathbb{Z}}

\newcommand{\B}{\mathbb{B}}
\newcommand{\I}{\mathbb{I}}
\newcommand{\CC}{\mathcal{C}}
\newcommand{\BB}{\mathcal{B}}
\newcommand{\HH}{\mathcal{H}}
\newcommand{\UU}{\mathcal{U}}

\newcommand{\Null}{\mathcal{N}}
\newcommand{\As}{\mathscr{A}}
\newcommand{\Us}{\mathscr{U}}
\newcommand{\Es}{\mathscr{E}}

\newcommand{\Hs}{\mathscr{H}}
\newcommand{\Bs}{\mathscr{B}}
\newcommand{\Ws}{\mathscr{W}}
\DeclareMathOperator{\rot}{rot}
\DeclareMathOperator{\inter}{int}
\DeclareMathOperator{\ind}{ind}
\DeclareMathOperator{\conv}{conv}

\newcommand{\mydeg}{\mathfrak{D}}

\newcommand{\inn}[1]{#1_\mathrm{in}}
\newcommand{\out}[1]{#1_\mathrm{out}}

\providecommand{\abs}[1]{\left\lvert#1\right\rvert}     
\newcommand{\scal}[2]{\left\langle #1,#2\right\rangle}

\newcommand{\clos}[1]{\overline{#1}}

\newtheorem{theorem}{Theorem}
\newtheorem{lemma}[theorem]{Lemma}
\newtheorem{prop}[theorem]{Proposition}
\newtheorem{corol}[theorem]{Corollary}
\theoremstyle{definition}
\newtheorem{defin}[theorem]{Definition}
\newtheorem{remark}[theorem]{Remark}
\newtheorem{example}[theorem]{Example}

\numberwithin{equation}{section}
\numberwithin{theorem}{section}

\usepackage{fancyhdr}
\title{A topological degree theory for rotating solutions of planar systems}

\author{\textsc{Paolo Gidoni}\\
	\small
	\textit{Dipartimento Politecnico di Ingegneria e Architettura, Università degli Studi di Udine}\\\small
	\textit{Via delle Scienze 206, 33100 Udine, Italy.} \texttt{paolo.gidoni@uniud.it}
}
\date{}

\begin{document}

	\maketitle
	
\begin{abstract}
	We present a generalized notion of degree for rotating solutions of planar systems. We prove a formula for the relation of such degree with the classical use of Brouwer's degree and obtain a twist theorem for the existence of periodic solutions, which is complementary to the Poincaré--Birkhoff Theorem. Some applications to asymptotically linear and superlinear differential equations are discussed.
\end{abstract}	
{\small
	}

\section{Introduction}

In this paper we develop a topological degree theory characterizing the rotational properties of evolution maps on the plane, formalizing and improving the strategy introduced with A. Margheri in \cite{GidMar}, with applications to the existence and multiplicity of periodic solutions for ordinary differential equations.

The relevance of the rotational features of a systems  in the search for periodic solutions has been well-known for a long time. Several specialized methods exploiting such properties have been developed, among which two main clusters can be identified. The first one is epitomised by the celebrated Poincaré--Birkhoff Theorem (see \cite{Reb} for an historical introduction).
The theorem shows the existence of two fixed-points for area-preserving maps on the annulus, satisfying a twist condition of the boundary. This result is particularly suited to the study of periodic solutions of periodic Hamiltonian flows, for which a generalization in higher dimension has been recently proposed by A.~Fonda and A. Ureña \cite{FonUre,FonUre2,FonGid1}. Yet, the area-preserving assumption is quite restrictive, albeit crucial for the existence of fixed-points. An attempt to circumvent this situation is given by the theory of \emph{bend-twist maps} \cite{DingT,PasZan}, where area-preservation is replaced by some structural assumption on the map along some curves in the interior of the annulus, producing a behaviour analogous to the one in the classical case. 

A second family of results concerns instead planar systems whose solutions starting from the boundary of a (topological) disk rotate in a suitable way, so that we can find a periodic solution starting inside the disk. Typically, this type of results is applied to systems which have a superlinear (or sublinear, or linear and non-resonant) behaviour at infinity.
In this sense, we mention the results in \cite{Fucik,Stru}, which use a fixed-point theorem based on Brouwer's degree, and in \cite{CaMaZa,Maw}, proposing a continuation technique of Leray--Schauder’s type. Although  classical applications of degree methods do not present any constraint such as area preservation, resulting in a wider field of applicability, they also do not directly provide any information on the rotation number, and therefore have limited use in multiplicity results (with some notable exceptions for area-preserving maps, cf.~\cite{MaReZa}).

With this work, we propose to incorporate the best aspects of both lines of investigation, combining Brouwer's degree  with the qualitative properties of the evolution map lifted in polar coordinates.
Denoting with $f_T(x)=\phi_T(x)-x$ the Poincaré map $\phi_T$ of the dynamics minus the identity, the classical approach is to consider  Brouwer's degree $\deg(f_T,U,0)$ on a suitable open set $U$. Here instead, assuming that the rotation of each solution starting from a point on the boundary $\partial U$ is well defined, we introduce a countable family of degrees $\mydeg(F_T,U,\nu_i)$, each associated with a specific number of rotations $i\in\Z$. The novel degree $\mydeg$ inherits some usual properties of a degree, such as additivity and invariance for admissible homotopies, and a non-zero degree $\mydeg$ implies the existence of a periodic solution with the corresponding rotation number.
 In Theorem \ref{th:degree}, we show that Brouwer's degree $\deg(f_T,U,0)$ is determined by the sum of all the degrees $\mydeg$. Hence, even when $\deg(f_T,U,0)=0$, some of the degrees $\mydeg$ might be non-zero; some examples can be found in Section~\ref{sec:appl}.

As illustrated by Theorems \ref{th:deg1}, results of the ``disk''-type described above can be embedded in the theory, and the role of rotation in such situation appears highlighted.
Still, our theory performs best when dealing with twist on an annulus $A=\out U\setminus \clos{\inn U}$. Our main results in such framework is Theorem \ref{th:main}, providing the existence of one or, in some cases, two periodic solutions with different rotation number. Comparing it to the Poincaré--Birkoff theorem, there are several differences, and, in a sense, the two results may be seen as complementing each other. 
Firstly, since our approach is degree-based, we do not required any structural assumption on the annulus (e.g., star-shaped or invariant boundary, as necessary in various versions of the Poincaré--Birkhoff Theorem), nor the area-preserving assumption, nor any replacement assumption in the interior as for bend-twist maps. 

Also the twist condition is different. In the Poincaré--Birkhoff Theorem, for some integer $n$, all solutions starting from the internal boundary $\partial \inn U$ must do less than $n$ rotations, and all solutions starting from the external boundary $\partial \out U$ must do more than $n$ rotations (or vice versa). The theorem then provides two periodic solutions making exactly $n$ rotations around the origin (cf.~Theorem \ref{th:PB}). Our twist condition \eqref{eq:twist} in Theorem \ref{th:main} instead allows an overlapping of the rotation ranges of the solutions starting respectively from the internal and from the external boundary, provided that no integer number is included in such region of overlapping.
However, an additional degree condition is necessary.  For each boundary of the annulus, if
the Brouwer's degree of $f_T(x)$ on $\inn U$ (resp.~on $\out U$) is different from one, then a periodic solution starting in the annulus is obtained; such solution has rotational properties similar to those of the corresponding boundary $\partial \inn U$ (resp.~$\partial\out U$). In particular, if $f_T$ has degree different from one on both $\inn U$ and $\out U$, two fixed points with different rotation properties are obtained.  We emphasize that \emph{Brouwer's degree on $\inn U$ and $\out U$ might also be the same} (provided it is not equal to one), implying a null degree on the annulus. This means that our method not only characterizes rotations, but is successful in situations where a standard application of Brouwer's degree would be totally ineffective. This situation is illustrated in Example \ref{ex:planar}.

The need of such an additional requirement is not surprising. Indeed, twist alone, in a non-Hamiltonian setting, is not sufficient to assure the existence of nontrivial periodic solutions in an annulus. A very wide family of examples is given by dissipative and expansive flows, with the origin as unique fixed point: in such situations, the degree of $f_T$ on an open set containing the origin is always equal to one.  An example of this is presented at the end of Section \ref{sec:fixedpoint}.  

Nevertheless, we notice that twist alone, without the additional degree condition, is sufficient to find a fixed-point of $\phi_T(x)$ in the whole disk $\out U$, but not necessarily in the annulus. This fact was already observed by W.Y.~Ding \cite{Ding} for a twist condition of Poincaré--Birkhoff-type; we prove it in Theorem \ref{th:ding} for our weaker twist condition.

The complementarity, in the area-preserving case, of Theorem \ref{th:main} and the Poin\-caré--Birkhoff Theorem is well illustrated by the case of planar Hamiltonian systems asymptotically linear at zero and at infinity, which we discuss in Section \ref{sec:asymp_lin}. In such a case, the two theorems provide the existence of periodic solutions with different rotation numbers.  More precisely, it is shown that, if the asymptotic linear systems are nonresonant with Maslov's indices respectively $i_0$ and $i_\infty$, then there exist at least $\abs{i_\infty-i_0}$ nontrivial $T$-periodic solutions, a fact conjectured by C.C. Conley and E. Zehnder in \cite{ConZeh}. This result was first proved in \cite{GidMar}, using a simplified version of the degree discussed here; we propose with Theorem \ref{th:gidmar} a shorter proof based on the twist Theorem \ref{th:main}.

The paper is structured as follows. In Section \ref{sec:defin}, we introduce the novel notion of degree $\mydeg$ and state some of its basic properties. Notice that the theory presented here holds for general open bounded sets, while the construction in \cite{GidMar} considered only balls centred in the origin. This improvement has not only a theoretical value, but also allows a wider range of applications, for instance the case of superlinear ODEs discussed in Section~\ref{sec:superlinear}.

In Section \ref{sec:comp} we prove the first main result of the paper, Theorem~\ref{th:degree}, consisting of a characterization of the degree $\mydeg$ in terms of rotational properties and Brouwer's degree of $f_T$. This provides a fast way to compute, or at least estimate, the degree $\mydeg$, making it much easier to employ in concrete problems. 

In Section \ref{sec:fixedpoint} we present some fixed-point theorems which follow from Theorem~\ref{th:degree}, including the twist Theorem \ref{th:main} discussed above. In Section \ref{sec:appl} we propose several applications of the theory to the existence and multiplicity of periodic solutions for asymptotically linear and superlinear systems.



\section{A new notion of degree for evolution maps in polar-like coordinates}\label{sec:defin}

\subsection{Polar-like coordinates and annuli} 
The first step, in order to study how solutions rotate, is to introduce suitable polar coordinates.
For our aims, we consider the following generalized notion of polar coordinates. Let us denote with $\Hs$ the open half-plane $\Hs:=\R\times (0,+\infty)$, and, accordingly, with $\clos \Hs$ its closure $\clos \Hs:=\R\times [0,+\infty)$. The origin of the plane is denoted as $0_{\R^2}$, but where there is no risk of ambiguity we will use a plain~$0$.

\begin{defin}[Polar-like coordinates] We say that a map $\Psi\colon \clos\Hs \to \R^2$ defines polar-like coordinates on $\R^2$ if it satisfies:
	\begin{itemize}
		\item (\emph{Periodicity}) $\Psi(\theta,r)$ is $2\pi$-periodic in the first variable $\theta$;
		\item (\emph{Origin}) $\Psi^{-1}(0_{\R^2})=\R\times\{0\}$;
		\item (\emph{Continuity}) $\Psi$ is continuous and its restriction to $\Hs$ is a local homeomorphism;
		\item (\emph{Bijectivity}) $\Psi$ defines a bijection between $[0,2\pi)\times (0,+\infty)$ and $\R^2\setminus\{0_{\R^2}\}$;
		\item  (\emph{Orientation}) the map $\Psi$ is orientation preserving.
	\end{itemize}
\end{defin}
We observe that the definition above, due to the orientation preserving assumption, applies to classical clockwise polar coordinates, but not to counterclockwise polar coordinates.

\begin{defin}[Annulus] \label{def:annulus}
	Let $\inn U$ and $\out U$, be two open, bounded subsets of $\R^2$, such that $0_{\R^2}\in\inn U\subset \clos{\inn U}\subset \out U$. We call (generalized) \emph{annulus} any set of the form $A=\clos{\out U}\setminus \inn U$.
\end{defin}

Notice that  our notion of annulus is more general than its classical meaning, where the sets $\out U,\inn U$ are concentric balls (or at least two simply connected sets).

Given a (generalized) annulus $A$, we denote with $\As$ the corresponding set $\As:=\Psi^{-1}(A)\subset \Hs$. Moreover, we write $\inn\Us:=\Psi^{-1}(\inn U)$ and $\out\Us:=\Psi^{-1}(\out U)$, so that $\As=\clos{\out\Us}\setminus \inn\Us$. 

Let us consider an open, bounded set $U$ with $0_{\R^2}\notin \partial U$ and set $\Us=\Psi^{-1}(U)$; we denote with $\partial U$ the boundary of $U$ and with $\partial \Us$ the boundary of $\Us$ relative to $\clos \Hs$, so that $\partial \Us=\Psi^{-1}(\partial U)$. 
For every $r>0$, we denote with $\Bs_r$ the strip $\Bs_r:=\R\times[0, r)$ and with $B_r$ the corresponding open set $B_r:=\Psi(\Bs_r)$. Notice that, for canonical polar coordinates, $B_r$ is the open ball of radius $r$ centred in the origin. As  the Reader might have observed, generally we will use a capital letter for sets in the plane $\R^2$ and its script version for their corresponding sets in the halfplane $\clos\Hs$. 

 The identity map is denoted with $\I$, possibly with a subscript indicating its domain. Moreover, given a continuous function $F\colon\Psi^{-1}(E)\to \R$ which is $2\pi$-periodic in $\theta$ (and constant on $\partial \HH$ if $0_{\R^2}\in E$), we  denote as $F\circ\Psi^{-1}\colon E\to \R^2$ the function which is more properly defined as $F\circ\left(\Psi\rvert_{\{(0,0)\}\cup[0,2\pi)\times (0,+\infty)}\right)^{-1}$.

We also observe that polar-like coordinates can be used to extend action-angle coordinates defined on a (classical) annulus, as those produced by a family of concentric periodic orbits of an autonomous planar Hamiltonian systems (of which one may study non-autonomous perturbations). In this case each component of boundary of the annulus may even not be star-shaped in $\R^2$, yet be expressed in suitable polar-like coordinates as $\partial B_r$; moreover, direction of the rotating component follows that of the periodic orbits, instead of being defined by canonical polar-coordinates. 

\medbreak

We denote with $\deg(g,U,v)$ Brouwer's topological degree of a continuous map $g\colon \clos U\to \R^2$ with respect to an open bounded set $U\subset \R^2$ and a value $v\in\R^2$ such that $v\notin g(\partial U)$.
Since the sets $B_r$ are homeomorphic to a ball, we have the following result.
\begin{prop} \label{prop:gradopalla}
	Let $P\colon \clos{B_r}\to \R^2$ be a continuous map, with $B_r\subset \R^2$ defined as above. If, for every $x\in \partial B_r$ we have $P(x)\in\clos{B_r}\setminus\{x\}$, then $\deg(P-\I,B_r,0)=1$.
\end{prop}

Given an isolated fixed point $x$ of $P$ with an isolating neighbourhood $U_x$, we denote its fixed-point index as $\ind(P,x):=\deg(\I-P,U_x,0)=\deg(P-\I,U_x,0)$, where the last equality holds since we are in dimension two.

\subsection{Evolution maps, periodic solutions and rotation} \label{sec:flow}

Let us consider  two  sets $E\subseteq \Omega \subseteq \R^2$  and the dynamics described by a continuous function $\phi\colon [0,T]\times \Omega\to \R^2$. The reader may visualize $\phi(t,x)$ as the Poincaré-time map, for instance, of a system of ordinary differential equations, that associates to each initial state $x$ at time zero the corresponding state at time $t$. Notice that, however, we require only that 
\begin{equation}\label{cond:flow}
	\text{$\phi$ is continuous and $\phi(0,\cdot)=\I$}
\end{equation} without any additional structure or regularity. This means for instance that we do not require backwards uniqueness, nor the structure of semi-dynamical system, since the orbits of different initial points may cross each other at a given time and state, but continue differently.

Indeed, the only role of the evolution map $\phi$ is to identify a specific lift $\Phi(T,\cdot)$ of the map $\phi(T,\cdot)$, restricted to $E$, to the halfplane $\Hs$, among all the possible lifts $\Phi(T,\cdot)+(2i\pi,0)$.
In order to do so, we assume that every trajectory starting from the set $E$ does not cross the origin of $\R^2$, namely
\begin{equation}\label{eq:lift_cond}
0_{\R^2}\notin\phi([0,T]\times E) \,.
\end{equation}
For convenience's sake, let us define the set $\Null\subset \R^2$ as
\begin{equation}\label{eq:defNull}
	\Null:=\{x\in \Omega : \phi(t,x)=0 \quad\text{for some $t\in[0,T]$}\}\,,
\end{equation}
so that condition \eqref{eq:lift_cond} reads $E \cap \Null=\emptyset$. 

Assuming \eqref{eq:lift_cond}, let us write $\Es=\Psi^{-1}(E)\subset \Hs$. We introduce $\Phi\colon [0,T]\times\Es\to\Hs$ as the unique continuous lift in polar-like coordinates  of the map $\phi$; its wellposedness $\Phi$ is proven in \cite[Theorem 11]{GidODE}. In other words, $\Phi=\Phi(t,(\theta,r))$ is the unique continuous map such that $\Phi(0,\cdot)=\I_\Es$ and
\begin{equation}
\Psi(\Phi(t,y))=\phi(t,\Psi(y)) \quad\text{for every $t\in[0,T],y\in\Es$} \,.
\end{equation}
In particular, this implies that 
\begin{equation}
\Phi(t,(\theta+2\pi,r))=\Phi(t,(\theta,r))+(2\pi,0) \quad\text{for every $(t,(\theta,r))\in[0,T]\times\Es$} \,.
\end{equation}

When there is no ambiguity, we will adopt the usual abbreviated notation $\phi_T,\Phi_T$ respectively for the functions $\phi(T,\cdot)$ and $\Phi(T,\cdot)$.

We are interested in studying the fixed points of $\phi_T$, corresponding to $T$-periodic solutions of the $T$-periodic extension in time of the map $\phi$. To do so, we introduce the function $f_T\colon A\to \R^2$ defined as $f_T(x):=\phi(T,x)-x$. Hence, the fixed points of $\phi_T$ coincide with the zeros of $f_T$.

We proceed similarly for the corresponding map $\Psi(T,\cdot)$. We therefore introduce the function $F_T\colon\Es\to \R^2$ defined as $F_T(y):=\Phi(T,y)-y$. For every index $i\in \Z$, we denote $\nu_i=(2i\pi,0)$. The map $F_T$ allows us to define the rotation performed by the trajectory of a point $x$ in the period $[0,T]$. 

\begin{defin}[Rotation]
	Let $E\subset \R^2$ be a subset such that $0_{\R^2}\notin E$ and set $\Es=\Psi^{-1}(E)\subseteq \Hs$. Let $F\colon \Es\to \Hs$ be a continuous map, $2\pi$-periodic in the first variable. We identify its two components as $F=:(F^\theta,F^r)$.  We define the rotation associated to the point $x\in E$ with respect to the map $F$ as
	\begin{equation*}
	\rot (F,x) := \frac{1}{2\pi}\,F^\theta(y)
	\end{equation*}
	  for any $y\in \Psi^{-1}(x)$. 
We denote the rotation of the whole set $E$ accordingly, as $\rot (F,E) :=\{\rot(F,x) : x\in E \}$. Moreover, we define the set $\Sigma (F,E)\subset \Z$ as
	\begin{equation*}
\Sigma (F,E) :=\Z\cap \conv \rot (F, E) \,,
	\end{equation*}
where $\conv$ denotes the convex envelope.
\end{defin}
We highlight that if $x$ is a fixed point of $\phi_T$, its rotation $\rot(F_T,x)$ is always an integer.
Often we will consider the case $E=\partial U$ for some open bounded set $U$; thus, since $\partial U$ is compact, also $\rot(F,\partial U)$ is compact. Notice however, that, if $\partial U$ is not connected, for instance when $U$ is the interior of an annulus, then also $\rot(F,\partial U)$ may not be connected. However,  $\conv \rot(F,\partial U)$ is always a compact interval, hence  $\Sigma(F,\partial U)$ has always a finite number (possibly zero) of elements, which are all consecutive integers.


\subsection{A  degree theory for evolution maps in polar-like coordinates}

In the following, we assume that $U\subset \R^2$ is an open bounded set such that $0_{\R^2}\notin\partial U$ and set $\Us=\Psi^{-1}(U)$. Notice that $\partial \Us=\Psi^{-1}(\partial U)$ and it is $2\pi$-periodic in the first variable, meaning $(\theta,r)\in \partial \Us$ if and only if $(\theta+2\pi,r)\in \partial \Us$.

\begin{defin} \label{def:mydeg}
	 Let $F\colon \partial\Us\to \R^2$ be a continuous function, $2\pi$-periodic in the first variable, and pick $v\in \R^2$. Assume that $F(\theta,r)\neq v$ for every $(\theta,r)\in \partial\Us$. We now pick any continuous function $F_\mathrm{ext}\colon \clos\Hs\to \R^2$, $2\pi$-periodic in the first variable, such that $F_\mathrm{ext}(\theta,r)\equiv F(\theta,r)$ for every $(\theta,r)\in \partial\Us$, and $F_\mathrm{ext}(\cdot,0)\equiv v_0$ for some constant $v_0\in\R^2$.
	We define the degree $\mydeg(F,r,v)$ as
\begin{equation*}
\mydeg(F,U,v):=\deg(F_\mathrm{ext}\circ \Psi^{-1},U,v)	 \,.
\end{equation*}	
\end{defin}

Let us observe that a continuous and periodic extension $F_\mathrm{ext}$ can always be constructed, via Tietze's extension Theorem applied to the closed subset $\partial U\cup \{0_{\R^2}\}\subset \R^2$ and the continuous function $g$ defined as $g(x)=F\circ \Psi^{-1}(x)$ for $x\in \partial U$ and such that $g(0)=v_0$. Moreover, the map $F_\mathrm{ext}\circ \Psi^{-1}\colon \R^2\to \R^2$ is a continuous function, so its topological degree is well defined. Notice also that, by the homotopy property of the topological degree, $\mydeg$ does not depend on the choice of the extension.

Let us now enunciate some properties of the new degree $\mydeg$ we have introduced, which follow as direct corollaries of the classical (corresponding) properties of Brouwer's degree, see, e.g., \cite{Deim,Fons}.

\begin{corol}[Homotopy] \label{cor:homot}
	Let $U\subset\R^2$ be an open bounded set such that $0_{\R^2}\notin \partial U$. Let $h\colon [0,1]\times \partial \UU  \to \R^2\setminus\{v\}$ be a continuous function, $2\pi$-periodic in $\theta$. Writing $F_0:=h(0,\cdot)$ and $F_1:=h(1,\cdot)$, we have
\begin{equation}
 \mydeg(F_0,U,v)=\mydeg(F_1,U,v) \,.
\end{equation}
\end{corol}

\begin{corol}[Additivity] Let $U,U_1,U_2$ be open bounded sets in $\R^2$ such that $U_1\subset U$, $U_2\subset U$, $U_1\cap U_2=\emptyset$ and $0_{\R^2}\notin W:=\clos U\setminus (U_1\cup U_2)$. Let $F\colon \Psi^{-1}(W)\to \R^2$ be a continuous function, $2\pi$-periodic in $\theta$, such that $F(\theta, r)\neq v$ for every $(\theta,r)\in\Psi^{-1}(W)$. Then
	\begin{equation}
		\mydeg(F,U,v)=\mydeg(F,U_1,v)+\mydeg (F,U_2,v) \,.
	\end{equation}
\end{corol}
As a direct application of additivity, we have the following result.
\begin{corol}\label{corol:annulus_degree}
	Given a generalized annulus $A=\clos{\out U}\setminus \inn U$, let $F\colon \As\to \R^2$ be a continuous function, $2\pi$-periodic in the first variable $\theta$. Suppose that  $v\notin F(\partial \out \Us)$ and $v\notin F(\partial \inn \Us)$. Then
	\begin{equation*}
	\deg(F\circ \Psi^{-1},A,v)=\mydeg(F,\out U,v)-\mydeg(F,\inn U,v) \,.
	\end{equation*}	
\end{corol}

Let us also observe the following fact:
\begin{lemma}\label{lemma:mydegrot} Let $U$ be an open bounded set such that $0_{\R^2}\notin \partial U$ and consider a continuous function $F\colon \partial\Us\to \R^2$, $2\pi$-periodic in the first variable. Then 
	\begin{equation}
	\mydeg(F,U,v)=0 \qquad\text{for every $v=(v^\theta,v^r)$ with $v^\theta\notin \conv\rot (F, \partial U)$} \,.
\end{equation}	
	In particular it holds
	\begin{equation}
		\mydeg(F,U,\nu_i)=0 \qquad\text{for every $i\notin\Sigma(F,\partial U)$} \,.
	\end{equation}
\end{lemma}
\begin{proof}
	First of all, let us notice that for every constant function $G\colon \partial \Us \to \R^2$ such that $G\equiv(G^\theta,G^r) \neq v$ we have
		\begin{equation}
		\mydeg(G,U,v)=0  \,.
	\end{equation}
Now set $G^\theta\in\conv\rot (F,\partial U)$ and consider the homotopy $h\colon [0,1]\times \partial\UU$ defined as
\begin{equation*}
	h(\lambda,y)=(1-\lambda)G+\lambda F(y) \,.
\end{equation*}
By the assumptions of the Lemma and the convexity of $\conv\rot (F,\partial U)$, we clearly have $v\notin h([0,1]\times \partial \Us)$, hence the thesis follows by Corollary \ref{cor:homot}.
\end{proof}

The winding number of $\Psi^{-1}\circ F_T$ along a path has been considered, for instance, in some proofs of the Poincaré--Birkhoff Theorem \cite{BroNeu,LeCal} and can be seen as a special case of $\mydeg(F,U,0)$ where $U$ is the region enclosed by a closed path. Our alternative and more abstract construction of $\mydeg(F,U,v)$ is however not a generalization for its own sake, since both the extension to a general open set $U$ and the possibility to simultaneously consider different values of $v$  will be pivotal for our main results and applications.

\section{A characterization of the $\mydeg$ degree} \label{sec:comp}

We now study in more detail the framework introduced in the previous section. In particular, we plan to investigate the relationship between the value of $\deg(f_T,U,0)$ and the values of $\mydeg(F_T,U,\nu_i)$, where we recall that $\nu_i:=(2\pi i,0)$. There is a substantial difference whether or not the origin belongs to $U$, as we see in the following theorem, cf.~also Figure \ref{fig:zero}.

\begin{theorem} \label{th:degree}
Let $U\subset \R^2$ be an open bounded set. Consider a map $\phi\colon[0,T]\times \clos U \to\R^2$ as in \eqref{cond:flow} and such that $\Null\cap \partial U=\emptyset$, i.e.~$0_{\R^2}\notin\phi([0,T]\times \partial U)$. The maps $f_T\colon\clos{U}\to\R^2$ and $F_T\colon\partial \Us \to\Hs$ are defined as in Section \ref{sec:flow}. Then
\begin{align}\label{eq:decomp}
	&\deg(f_T,U,0)=1+\sum_{i\in\Sigma(F_T,\, \partial U)}\mydeg(F_T,U,\nu_i) &&\text{if $0_{\R^2}\in U$} \,;\\[3mm]
	&\deg(f_T,U,0)=\sum_{i\in\Sigma(F_T,\, \partial U)}\mydeg(F_T,U,\nu_i) &&\text{if $0_{\R^2}\notin \clos U$} \,.
\end{align}
\end{theorem}

\begin{figure}[tb]
	\centering 
	\begin{tikzpicture}[line cap=round,line join=round,>=stealth,x=1.0cm,y=1.0cm, line width=1.1pt, scale=0.8]
		\clip(-3.,-3.) rectangle (9.,5.);
		\draw [line width=0.8pt] (0.,0.) circle (2.cm);
		\draw [line width=0.8pt] (5.,3.) circle (1.cm);
		\draw [->,color=ForestGreen] (0.,2.) -- (1.41,1.41);
		\draw [->,color=ForestGreen] (1.73,1.) -- (1.93,-0.52);
		\draw [->,color=ForestGreen] (1.73,-1.) -- (0.52,-1.93);
		\draw [->,color=ForestGreen] (0.,-2.) -- (-1.41,-1.41);
		\draw [->,color=ForestGreen] (-1.73,-1.) -- (-1.93,0.52);
		\draw [->,color=ForestGreen] (-1.73,1.) -- (-0.52,1.93);
		\draw [->,color=ForestGreen] (5.80,3.6) -- (6.65,-1.56);
		\draw [->,color=ForestGreen] (5.92,2.61) -- (6.03,-2.34);
		\draw [->,color=ForestGreen] (5.12,2.01) -- (5.04,-2.2);
		\draw [->,color=ForestGreen] (4.2,2.4) -- (4.67,-1.27);
		\draw [->,color=ForestGreen] (4.08,3.39) -- (5.29,-0.49);
		\draw [->,color=ForestGreen] (4.88,3.99) -- (6.28,-0.63);
		\draw [->,color=Dandelion] (0.,2.) -- (1.4,2.);
		\draw [->,color=Dandelion] (1.73,1.) -- (3.13,1.);
		\draw [->,color=Dandelion] (1.73,-1.) -- (3.13,-1.);
		\draw [->,color=Dandelion] (0.,-2.) -- (1.4,-2.);
		\draw [->,color=Dandelion] (-1.73,-1.) -- (-0.33,-1.);
		\draw [->,color=Dandelion] (-1.73,1.) -- (-0.33,1.);
		\draw [->,color=Dandelion] (4.88,3.99) -- (6.28,3.99);
		\draw [->,color=Dandelion] (5.8,3.6) -- (7.2,3.6);
		\draw [->,color=Dandelion] (5.92,2.61) -- (7.32,2.61);
		\draw [->,color=Dandelion] (5.11,2.01) -- (6.52,2.01);
		\draw [->,color=Dandelion] (4.2,2.4) -- (5.6,2.4);
		\draw [->,color=Dandelion] (4.08,3.39) -- (5.48,3.39);
		\draw [fill=black] (0.,0.) circle (1pt) node[below]{$0_{\mathbb R^2}$};
		\draw (-1.8,1.8) node {$U_0$};
		\draw (4.1,4.1) node {$U_1$};
		\draw [color=Dandelion] (7.5,3.1) node {$F_T\circ\Psi^{-1}$};
		\draw [color=ForestGreen] (7,0) node{$f_T$};
		\draw [color=Dandelion] (1,2.4) node {$F_T\circ\Psi^{-1}$};
		\draw [color=ForestGreen] (0.7,1.2) node{$f_T$};
	\end{tikzpicture}
	\captionsetup{singlelinecheck=off}
	\caption[.]{An illustration of the difference between the two alternative cases in Theorem \ref{th:degree}, represented by the two sets $U_0\ni 0_{\R^2}$ and $U_1\not\ni 0_{\R^2}$, for an evolution map $\phi$ rotating the plane around the origin clockwise by $\pi/4$ radians. The map $f_T$, drawn in green, satisfies $\deg(f_T,U_0,0)=1$ and $\deg(f_T,U_1,0)=0$. The map $F_T\circ \Psi^{-1}$, in yellow, has constant value $(\pi/4,0)$, so that $\deg(F_T\circ\Psi^{-1},U_0,0)=\mydeg(F_T,U_0,0)=0$ and $\deg(F_T\circ\Psi^{-1},U_1,0)=\mydeg(F_T,U_1,0)=0$. Notice, in particular, that $\Sigma(F_T,\partial U_0)=\Sigma(F_T,\partial U_1)=\emptyset$.}
	\label{fig:zero}
\end{figure}

\begin{proof} 
	We consider first the case $0_{\R^2}\in U$. In such a case, there exists a positive radius $r_0>$ such that $\clos{B_{r_0}}\subset U$. We define $A_0$ as the (generalized) annulus $A_0=\clos U\setminus B_{r_0}$, write $\As_0=\Psi^{-1}(A_0)$ and recall that $\clos{\BB_{r_0}}=\R\times [0,r_0]$. We also introduce the compact interval $I_\mathrm{rot}$ as
	\begin{equation*}
		I_\mathrm{rot}:=\begin{cases}
			[j-\frac{1}{2},j+\frac{1}{2}] &\text{if $\conv\rot(F_T,\partial U)= \Sigma(F_T,\partial U)=\{j\}$}\subset\Z \,;\\
			\conv\rot(F_T,\partial U) &\text{otherwise} \,.
		\end{cases}
	\end{equation*}
	
	We claim that there exists a continuous function $\widetilde F\colon \clos{\Us}\to \R^2$ satisfying	
	\begin{enumerate}[label=\textup{\alph*)}]
		\item $\widetilde F(\theta,r)= F_T(\theta, r)$ for every $(\theta,r)\in \partial \Us$; \label{cond:Fext1}
		\item $\widetilde F(\theta,r)$ is $2\pi$-periodic in $\theta$;  \label{cond:Fext2}
		\item $\rot(\widetilde{F},U)\subseteq I_\mathrm{rot}$; \label{cond:Fext3}
		\item there exists a positive constant $c_F>0$ such that $\widetilde F^r(\theta,r)+r\geq c_F$ for every $(\theta,r)\in \clos{\As_0}$; \label{cond:Fext4}
	\item $\widetilde F(\theta,r)=(2\pi \alpha,0)$ for every $(\theta,r)\in \clos{\BB_{r_0}}$, where  $\alpha\in \R$ satisfies $\alpha\in I_\mathrm{rot}\setminus \Sigma(F_T,\partial U)$. \label{cond:Fext5}
\end{enumerate}
The existence of such a function is a consequence of Tietze's extension Theorem. More precisely, we extend from the closed subset  $\partial A_0\subset A_0$ to $A_0$ the function $x=\Psi(\theta,r)\mapsto \widetilde F(\theta,r)+(0,r)$, defined using \ref{cond:Fext1} and \ref{cond:Fext5}.  Properties \ref{cond:Fext3} and \ref{cond:Fext4} follow from the fact that the extension may be taken with the same bounds of the function on $\partial A_0$; in particular notice that we extend $\widetilde F(\theta,r)+(0,r)$ instead of $\widetilde F$ in order to preserve the bound \ref{cond:Fext4}, which holds by construction on $\partial A_0$. The periodicity \ref{cond:Fext2} comes from the fact that we performed the extension on $\clos U$ instead of $\clos \Us$. 

Let us now define the continuous functions $\widetilde P\colon \clos \Us\to \clos\Hs$,
$\tilde p\colon \clos U \to \R^2$, $\tilde f\colon \clos U \to \R^2$  as
\begin{align*}
	\widetilde P(y)=y+\widetilde F(y)\,, && \widetilde p(x)=\Psi \circ \widetilde{P} \circ \Psi^{-1}(x)\,, && \tilde f(x)=\tilde p(x)-x \,.
\end{align*}
Notice that $\tilde p$ is well-defined by \ref{cond:Fext2}; moreover for every $x\in \partial U$ we have $\tilde p(x)=\phi(T,x)$, and therefore $\tilde f(x)=f_T(x)$. 

We also observe that $\tilde f$ is continuous, $\tilde f(x)\neq 0$ for every $x\in \clos{B_{r_0}}\setminus\{0_{\R^2}\}$ and, by \ref{cond:Fext5}, $\tilde p(\clos{B_{r_0}})\subseteq\clos{B_{r_0}}$. Hence, by Proposition \ref{prop:gradopalla}, we deduce  that $\deg(\tilde f,B_{r_0},0)=1$.
Therefore, by the additivity property of the degree, we obtain
\begin{align}\label{eq:mydegdec0}
	\deg(f_T,U,0)&=\deg(\tilde f,U,0)=\deg(\tilde f,\inter A_0,0)+\deg(\tilde f,B_{r_0},0)\notag\\
	&=\deg(\tilde f,\inter A_0,0)+1 \,.
\end{align}

We now compute $\deg(\tilde f,\inter A_0,0)$.
We claim that, for every $\epsilon>0$, there exists $\hat f\in\CC^1(A_0,\R^2)$ such that
\begin{itemize}
	\item $0$ is a regular value for $\hat f$; in particular $\hat f^{-1}(0)$ consists of a finite number of isolated points;
	\item $\abs{\hat f -\tilde{f}}<\epsilon$ uniformly on $A_0$.
\end{itemize} 
This follows from the approximability of continuous function by differentiable functions, combined with Sard's Lemma, see for instance \cite[Prop.~1.19]{Fons}. The fact that $\hat f^{-1}(0)$ is finite follows by compactness, cf.~\cite[Exer.~1.2]{Fons}  We remark that the existence of such an $\hat f$ is the key point of one of the classical constructions of Brouwer's degree for a continuous function, cf.~for instance \cite{Deim,Fons,Nag}.

We set $\epsilon<\min \{\lvert\tilde f(x)\rvert, x\in\partial A_0\}$ so that 
\begin{equation}\label{eq:homot_hatf}
 (1-\lambda)\tilde f(x)+\lambda \hat f(x)\neq 0\qquad \text{for every $(\lambda,x)\in  [0,1]\times\partial A_0$}\,, 
\end{equation}
which implies $\deg(\hat f,\inter A_0,0)=\deg(\tilde f,\inter A_0,0)$.
Moreover,  by \ref{cond:Fext4} we deduce that there exists $\tilde \epsilon>0$ such that $\abs{\tilde{p}(x)}>2\tilde \epsilon$ for every $x\in A_0$. We also set $\epsilon<\tilde \epsilon$, so that the homotopy $h\in\CC([0,1]\times A_0,\R^2)$ defined as
\begin{equation*}
h(\lambda,x):=(1-\lambda)\tilde{p}(x)+\lambda \hat p(x)
\end{equation*}
satisfies $\abs{h(\lambda,x)}>\tilde{\epsilon}$ for every $(\lambda,x)\in [0,1]\times A_0$. We define $H\colon [0,1]\times \As_0\to \Hs$  as the unique homotopy such that $H(0,\cdot)=\widetilde{P}(\cdot)$ and
\begin{equation*}
\Psi (H(\lambda,y))= h(\lambda,\Psi(y)) \quad \text{for every $(\lambda,y)\in [0,1]\times \As_0$\,.}
\end{equation*}
The wellposedness of $H$ can be verified by a simple adaptation of \cite[Theorem 11]{GidODE} obtained setting there $g(\tau,\eta):=h(\tau,\Psi(\eta))$.  By \eqref{eq:homot_hatf} we deduce that $H(\lambda,y)\neq y+\nu_i$ for every $y\in \partial \As_0$, $\lambda\in[0,1]$ and $i\in \Z$. We also denote $\widehat{P}:=H(1,\cdot)$ and $\widehat F:=\widehat{P}-\I$.

Let us pick any $y_x\in \Psi^{-1}(x)$. We observe that the following statements are equivalent:
\begin{itemize}
	\item $x\in \inter A_0$ is a zero of $\hat{f}$\,;
	\item $x\in \inter A_0$ is a fixed point of $\hat{p}$\,;
	\item $y_x\in \inter A_0$ is a fixed point of $\widehat{P}-\nu_{i_x}$, for a (unique) index $i_x\in \Z$\,;
	\item   $y_x\in \inter A_0$ satisfies $\widehat{F}(y_x)=\nu_{i_x}$, for the same index $i_x\in \Z$\,;
	\item $x\in \inter A_0$ satisfies $\widehat{F}\circ \Psi^{-1}(x)=\nu_{i_x}$, for the same index $i_x\in \Z$ \,.
\end{itemize} 

If $x$ is an isolated zero for $\hat f$, then there exists two neighbourhoods $V_x\subseteq W\subseteq \inter A_0$ such that $x$ is the only  zero for $\hat f$ in $V_x$ and $\hat p(V_x)\subseteq W$. Moreover, without loss of generality, we can take $V_x,W$ sufficiently small such that, for a suitable choice of a parameter $a$, there exists a neighbourhood $\Ws\subset \inter \As_0\cap \bigl((a,a+2\pi)\times(0,+\infty)\bigr)$ such that $\Psi\rvert_\Ws\colon \Ws\to W$ is a homeomorphism between $\Ws$ and $W$. Let us set $y_x=\Psi\rvert_\Ws^{-1}(x)$.

We claim that 
\begin{equation} \label{eq:fixedpointdeg}
	\deg(\hat f,V_x,0)=\deg(\widehat{F}\circ \Psi^{-1},V_x,\nu_{i_x}) \,.
\end{equation}
Indeed, recalling that $\ind(p,x)$ denotes the fixed-point index, we have
\begin{align*}
\deg(\hat f,V_x,0)=\ind (\hat p, x) \,, && \deg(\widehat F,\Psi\rvert_\Ws^{-1}(V_x),\nu_{i_x})=\ind(\widehat{P}-\nu_{i_x},y_x) \,.
\end{align*}
Please notice that we are using the fact that we are in dimension two, since the fixed point index of a map $p$ is more properly associated to the degree of $\I-p$. Since $\widehat{P}-\nu_{i_x}= \Psi\rvert_\Ws^{-1}\circ \hat p \circ \Psi\rvert_\Ws$ on $\Ws$, by the commutativity  of the fixed-point index (see, e.g., \cite[Ch.~12]{Dug}) we obtain
\begin{equation*}
\ind (\hat p, x)=\ind (\Psi\rvert_\Ws\circ\Psi\rvert_\Ws^{-1} \circ \hat{p}  , x)=\ind (\Psi\rvert_\Ws^{-1} \circ \hat{p} \circ \Psi\rvert_\Ws, y_x)=\ind(\widehat{P}-\nu_{i_x},y_x) \,.
\end{equation*}
By the multiplication formula for the degree (cf., e.g., \cite[Theorem 2.10]{Fons}), using also the fact that $\Psi\rvert_\Ws$ is an orientation preserving homeomorphism, we deduce
\begin{align*}
\deg(\widehat F,\Psi\rvert_\Ws^{-1}(V_x),\nu_{i_x})=\deg(\widehat F\circ\Psi\rvert_\Ws^{-1},V_x,\nu_{i_x}) \,.
\end{align*}
Since $\widehat F\circ\Psi\rvert_\Ws^{-1}$ and $\widehat F\circ\Psi^{-1}$ coincide on $V_x$, connecting the previous three equalities we obtain \eqref{eq:fixedpointdeg}.

By the additivity property of the degree and the properties of $\hat f$, we know that
\begin{equation}\label{eq:proofdeg0}
\deg(\tilde f,\inter A_0,0)=	\deg(\hat f,\inter A_0,0)=\sum_{\substack{x\in\inter A_0\\ \hat f(x)=0}}\deg (\hat{f},V_x,0)
\end{equation}
where the isolating neighbourhoods $V_x$ are defined as above. By this and \eqref{eq:fixedpointdeg}  we deduce
\begin{align}
\sum_{\substack{x\in\inter A_0\\ \hat f(x)=0}}\deg (\hat{f},V_x,0) &=\sum_{i\in\Z}\sum_{\substack{x\in\inter A_0\\ \widehat F\circ\Psi^{-1}(x)=\nu_i}} \deg(\widehat F\circ\Psi^{-1},V_x,\nu_{i}) \notag\\
&=\sum_{i\in\Z} \deg(\widehat F\circ\Psi^{-1},\inter A_0,\nu_{i})\label{eq:proofdeg1} \,.
\end{align}
We also know that, for every $i\in \Z$,
\begin{equation}\label{eq:proofdeg2}
\deg(\widehat F\circ\Psi^{-1},\inter A_0,\nu_{i})=\deg(\widetilde F\circ\Psi^{-1},\inter A_0,\nu_{i})=\deg(\widetilde F\circ\Psi^{-1},U,\nu_{i}) \,.
\end{equation}
Indeed, since by \ref{cond:Fext5} the function $\widetilde F\circ\Psi^{-1}$ is constant on $\clos{B_{r_0}}$ with a value different from any $\nu_i$, we have $\deg(\widetilde F\circ\Psi^{-1},B_{r_0},\nu_{i})=0$, so that the last equality in \eqref{eq:proofdeg2} is obtained by the additivity of the degree.

By \eqref{eq:mydegdec0}, \eqref{eq:proofdeg0}, \eqref{eq:proofdeg1}, \eqref{eq:proofdeg2}, and Definition \ref{def:mydeg} we obtain
\begin{equation}
\deg(f_T,U,0)=1+\sum_{i\in\Z}\mydeg(F_T,U,\nu_i) \,.
\end{equation}
To obtain \eqref{eq:decomp}, we simply notice that, by Lemma \ref{lemma:mydegrot}, we deduce that 
\begin{equation*}
	\mydeg(F_T,U,\nu_i)=0 \qquad\text{for every $i\notin\Sigma(F_T, U)$} \,.
\end{equation*}
The proof in the case $0_{\R^2}\in U$ is thus concluded.

\medskip
Let us now assume that $0_{\R^2}\notin\clos U$. The proof is the same as in the previous case, but with a simplification in the argument.  The key point is that there is no longer the need to build the annulus $A_0$, but we can work directly on $\clos U$ since the origin is already excluded. Hence, in the construction of $\tilde f$ we require only \ref{cond:Fext1},\ref{cond:Fext2},\ref{cond:Fext3} and 
 	\begin{enumerate}[label=\textup{\alph*')}, start=4]
	\item there exists a positive constant $c_F>0$ such that $\widetilde F^r(\theta,r)+r\geq c_F$ for every $(\theta,r)\in \clos\Us$; \label{cond:Fext4p}
\end{enumerate}
while instead of the second part of \eqref{eq:mydegdec0} we just have $\deg(f_T,U,0)=\deg(\tilde f,U,0)$, since we miss the term $+1$ produced by the neighbourhood of the origin. From \eqref{eq:mydegdec0} onwards, the proof is exactly the same as above, replacing $\inter A_0$ with $U$ in each occurrence.
\end{proof}


\section{Fixed-point theorems for planar system with rotational or twist properties}\label{sec:fixedpoint}

In this section we propose some fixed point theorems that can be obtained using the degree~$\mydeg$.
The first one is a direct corollary of Theorem \ref{th:degree}.
\begin{theorem} \label{th:deg1}
	Let us consider an open bounded set $U\subset\R^2$ such that $0_{\R^2}\in U$. Let  $\phi\colon[0,T] \times\clos U\to\R^2$ be a continuous evolution map  such that $\Null\cap \partial U=\emptyset$ and define the maps $f_T\colon\clos{U}\to\R^2$ and $F_T\colon\partial\Us\to\R^2$ as in Section \ref{sec:flow}. 
	
	If $\Sigma(F_T,\partial U)=\emptyset$, then $\deg(f_T,U,0)=1$, and thus $\phi_T$ has a fixed point in~$U$.
\end{theorem}

As a direct consequence of Theorem \ref{th:deg1}, we observe that, if $\deg(f_T,U,0)\neq1$, then  $\Sigma(F_T,\partial U)\neq\emptyset$.

As we show in Theorem \ref{th:GidODE}, Theorem \ref{th:deg1} can be used to prove the existence of a periodic solution for planar systems whose solutions, starting at a sufficiently large radius, make an arbitrary large number of rotations.

For the next theorems, we will consider twist on a generalized annulus  $A=\clos{\out U}\setminus \inn U$.

\begin{theorem}\label{th:main} Let  $\phi\colon[0,T] \times A\to\R^2$ be a continuous evolution map  such that $\Null\cap \partial A=\emptyset$ and define the maps $f_T\colon\clos{U}\to\R^2$ and $F_T\colon\partial\Us\to\R^2$ as in Section \ref{sec:flow}. 	
Suppose that the twist condition 
\begin{equation}\label{eq:twist}
\Sigma(F_T,\partial \out U)\cap \Sigma(F_T,\partial \inn U)=\emptyset
\end{equation}
is satisfied. Then
\begin{itemize}
	\item if $\deg(f_T,\inn U,0)\neq 1$, then there exists a fixed point $x\in A$ of $\phi_T$ such that $F_T(x)=\nu_i$ for some $i\in \Sigma(F_T,\partial \inn U)$.
	\item if $\deg(f_T,\out U,0)\neq 1$, then there exists a fixed point $x\in A$ of $\phi_T$ such that $F_T(x)=\nu_i$ for some $i\in \Sigma(F_T,\partial \out U)$;
\end{itemize}
\end{theorem}
\begin{proof}
	Let us assume that $\deg(f_T,\inn U,0)\neq 1$. Then, by Theorem \ref{th:deg1}  there exists $i\in \Sigma(F_T,\partial \inn U)$ such that $\mydeg(F_T, \inn U,\nu_i)\neq 0$. On the other hand, by \eqref{eq:twist} and Lemma \ref{lemma:mydegrot}, we deduce that $\mydeg(F_T, \out U,\nu_i)=0$. By Corollary \ref{corol:annulus_degree} we deduce
	\begin{equation*}
		\deg(F_T\circ \Psi^{-1},A,\nu_i)\neq 0
	\end{equation*}
	which implies the existence of a fixed point $x\in A$ of $\phi_T$ such that $F_T(x)=\nu_i$.
	
	The case $\deg(f_T,\out U,0)\neq 1$ is proved analogously.
\end{proof}

Using a result of Simon \cite{Sim}, stating that the fixed-point index of an isolated fixed point of an area-preserving planar map is less or equal to one, we have the following corollary which can be applied, for instance, to the evolution maps of Hamiltonian systems.

\begin{corol} \label{corol:areapres}
	In the framework of Theorem \ref{th:main}, suppose in addition that the evolution map $\phi$ is an area-preserving map.  Then
	\begin{itemize}
		\item if $\deg(f_T,\out U,0)=h>1$, then there exist at least $h-1$ fixed points $x_l\in A$ of $\phi_T$ such that $F_T(x_l)=\nu_i$ for some $i\in \Sigma(F_T,\partial \out U)$;
		\item if $\deg(f_T,\inn U,0)=-k< 1$, then there exist at least $k+1$ fixed points $x_l\in A$ of $\phi_T$ such that $F_T(x)=\nu_i$ for some $i\in \Sigma(F_T,\partial \inn U)$.
	\end{itemize} 
\end{corol}
\begin{proof}
The proof proceed as that of Theorem \ref{th:main}. The multiplicity follows by combining the result of \cite[Prop.~1]{Sim} recalled above, so that $\deg(F_T\circ \Psi^{-1},A,\nu_i)=n>0$ implies the existence of at least $n$ geometrically distinct fixed points of $\Phi_T-\nu_i$.
\end{proof}

\begin{remark}
	The key point of our approach is to find some open bounded set $U$ and $i\in \Z$ such that $\mydeg(F_T, U,\nu_i)\neq 0$. The most direct way to recover such a set is to find $U$ such that
	$\deg(f_T,U,0)\neq 1$ and apply Theorem \ref{th:degree}.  If instead we have $\deg(f_T,U,0)= 1$ our target becomes much more unlikely. Firstly, from Theorem \ref{th:deg1}, we know that all the unfavourable situations with $\Sigma(F,\partial U)=\emptyset$ fall in the $\deg(f_T,U,0)= 1$ case. Secondly, if $\Sigma(F,r)=\{j\}$, then $\deg(f_T,U,0)= 1$ implies, by Theorem \ref{th:degree}, that $\mydeg(F_T, U,\nu_j)= 0$. Hence, when the degree is one, favourable situations may appear only if $\Sigma(F,\partial U)$ has at least two elements. Even in such a case, it is not automatic to have $\mydeg(F_T, U,\nu_i)\neq 0$, since all the degrees $\mydeg(F_T, U,\nu_i)$ may be zero. For instance, we can consider the map $\Phi(t,(\theta,r))=(\theta+q(\theta)t,r+q(\theta)t)$ where $q(\theta)=2+\sin\theta$. This means that, in the plane, the orbits are the arms of a symmetric spiral, but along some arms the system evolves faster. For $T$ sufficiently large we may obtain an arbitrarily large $\Sigma(F,\partial U)$, but by a symmetry argument one may observe that $\mydeg(F_T, U,\nu_i)= 0$ for every $i$.
\end{remark}

To conclude this section, we present present a last abstract theorem, which generalizes in the planar case a fixed-point Theorem by W.Y. Ding \cite{Ding}, cf.~also \cite{CheRen,CheRen2} for applications and \cite{LQ20} for an extension in higher dimension of that result.

\begin{theorem} \label{th:ding} Let  $\phi\colon[0,T] \times\clos{\out U}\to\R^2$ be a continuous evolution map  such that $\Null\cap \partial A=\emptyset$ and define the maps $f_T\colon\clos{\out U}\to\R^2$ and $F_T\colon\partial\As\to\R^2$ as in Section~\ref{sec:flow}.
Assume that the twist condition \eqref{eq:twist} holds. Then $\phi_T$ has a fixed point in $\clos{\out U}$.
\end{theorem}	

\begin{proof}
We proceed by cases:
\begin{itemize}
	\item if $\deg(f_T,\out U,0)$ or  $\deg(f_T,\inn U,0)$ is not defined, it means that we have the desired period solution starting respectively from $\partial \out U$ or $\partial \inn U$;
	\item if $\deg(f_T,\out U,0)\neq 0$ or  $\deg(f_T,\inn U,0)\neq 0$, it is a trivial application of degree theory;
	\item if $\deg(f_T,\out U,0)=\deg(f_T,\inn U,0)=0$, then $\Sigma(F,\partial\out{U})$ and $\Sigma(F,\partial \inn U)$ are both non-empty.  Since the twist condition \eqref{eq:twist} holds, we apply Theorem \ref{th:main}, recovering two fixed points for $\phi_T$. 
\end{itemize}\vspace{-4mm}
\end{proof}
Notice that twist alone, as in Theorem \ref{th:ding}, is not sufficient to provide existence of a fixed point in the annulus, as happens instead in Theorem \ref{th:main}. Consider for instance the evolution described, in polar-like coordinates, by the map $\Phi(t,(\theta,r))=(\theta +rt,r(1+t))$. The twist condition \eqref{eq:twist} is satisfied when the annulus is the difference of two balls $B_r$ with different radius, but the only fixed point for $\phi_T$ is the origin $0_{\R^2}$. Moreover, it is easy to verify that in this case $\deg (f_T,B_r,0)=1$ for every $r>0$, hence, as expected, we cannot apply Theorem \ref{th:main}.



\section{Applications and examples} \label{sec:appl}
For simplicity, in this section we choose as polar-like coordinates $\Psi$  the usual clockwise polar coordinates, namely
\begin{equation}\label{eq:Fpolar}
	\Psi(\theta,r)=(r\cos\theta, -r \sin \theta) \,.
\end{equation}
We denote with $\scal{z_a}{z_b}$ the scalar product of two vectors $z_a,z_b\in\R^2$. Given a smooth function $h(t,z)\colon [0,T]\times\R^2\to\R^2$, we write its differential in the $z$ variables as $D_zh$.

\subsection{Asymptotically linear planar systems} \label{sec:asymp_lin}
A classical situation in which the relevant rotational and degree properties can be easily studied is when the dynamics is produced by a system $\dot z=h(t,z)$ asymptotically linear at the origin or at infinity. First, we recall the following two lemmata, relating the topological properties of the system to those of its linearizations, cf.~\cite[Lemmata 1, 2 and 3]{MaReZa}. Such results are stated for a planar system $\dot z=h(t,z)$, where $h\colon\R\times \R^2\to\R^2$ is continuous, $T$-periodic in $t$ and continuously differentiable in $z$. We recall that a linear system is said to be $T$-nonresonant if it does not admit any nontrivial $T$-periodic solution.

\begin{lemma}\label{lem:asymp_zero}
	Assume that $h(t,0)=0$ and $D_z h(t,z)=L_0(t)$ for every $t\in\R$, with the system $\dot z=L_0(t)z$ being $T$-nonresonant. Then for every $\epsilon>0$ there exists $\bar r_0=\bar r_0(\epsilon)$ such that, for every $0<r_0<\bar r_0$ we have
	\begin{gather*}
		\deg (f_T,B_{r_0},0)=\deg (f_T^{L_0},B_{r_0},0) \\\rot(F_T,\partial B_{r_0})\subset \, (-\epsilon,\epsilon)+\rot(F_T^{L_0},\partial B_{r_0})
	\end{gather*}
	where, with their usual meaning from Section \ref{sec:flow}, $f_T,F_T$ are associated to the nonlinear system $\dot z=h(t,z)$, while $f_T^{L_0},F_T^{L_0}$ to its linearization at zero $\dot z=L_0(t)z$. The sum in the second equation denotes the Minkowski sum of two intervals in $\R$.
	Moreover, we notice that, taking $\epsilon$ sufficiently small,  we obtain $\Sigma(F_T,\partial B_{r_0})=\Sigma(F_T^{L_0},\partial B_{r_0})$.
\end{lemma}

\begin{lemma}\label{lem:asymp_inf}
	Assume that
	\begin{equation*}
		\lim_{\abs{x}\to +\infty} \frac{\abs{h(t,x)-L_\infty(t)(x)}}{\abs{x}}=0 \quad\text{uniformly in $t\in[0,T]$}
	\end{equation*}
	for a continuous matrix $L_\infty(t)$ such that the system $\dot z=L_\infty(t)z$ is $T$-nonresonant. Then, for every $\epsilon>0$ there exists $\bar r_\infty=\bar r_\infty(\epsilon)$ such that, for every $r_\infty>\bar r_\infty$ we have
	\begin{gather*}
		\deg (f_T,B_{r_\infty},0)=\deg (f_T^{L_\infty},B_{r_\infty},0) \\\rot(F_T,\partial B_{r_\infty})\subset \, (-\epsilon,\epsilon)+\rot(F_T^{L_\infty},\partial B_{r_\infty})
	\end{gather*}
	where, with their usual meaning from Section \ref{sec:flow}, $f_T,F_T$ are associated to the nonlinear system $\dot z=L_0(t)z$, while $f_T^{L_\infty},F_T^{L_\infty}$ to its linearization at infinity $\dot z=L_\infty(t)z$.	
	Moreover, we notice that, taking $\epsilon$ sufficiently small,  we obtain  $\Sigma(F_T,\partial B_{r_\infty})=\Sigma(F_T^{L_\infty},\partial B_{r_\infty})$.
\end{lemma}

\begin{figure}[p]
	\centering 
	\begin{tikzpicture}[line cap=round,line join=round,>=stealth,x=1cm,y=1cm, line width= 1.1pt, scale=0.75]
		\clip(-6,-4.5) rectangle (10.,4.5);
		\draw [line width=0.8pt] (0,0) circle (1cm);
		\draw [->, color=ForestGreen] (1,0) -- (0.717,0);
		\draw [->, color=ForestGreen] (0.866,0.5) -- (0.621,0.698);
		\draw [->, color=ForestGreen] (0.5,0.866) -- (0.358,1.209);
		\draw [->, color=ForestGreen] (0,1) -- (0,1.396);
		\draw [->, color=ForestGreen] (-1,0) -- (-0.717,0);
		\draw [->, color=ForestGreen] (-0.866,0.5) -- (-0.621,0.698);
		\draw [->, color=ForestGreen] (-0.5,0.866) -- (-0.358,1.209);
		\draw [->, color=ForestGreen] (0.866,-0.5) -- (0.621,-0.698);
		\draw [->, color=ForestGreen] (0.5,-0.866) -- (0.358,-1.209);
		\draw [->, color=ForestGreen] (0,-1) -- (0,-1.396);
		\draw [->, color=ForestGreen] (-0.5,-0.866) -- (-0.358,-1.209);
		\draw [->, color=ForestGreen] (-0.866,-0.5) -- (-0.621,-0.698);
		\draw [line width=0.8pt] (0,0) circle (3cm);
		\draw [->, color=ForestGreen] (3,0) -- (2.150,0);
		\draw [->, color=ForestGreen] (2.598,1.5) -- (1.862,2.093);
		\draw [->, color=ForestGreen] (1.5,2.598) -- (1.075,3.626);
		\draw [->, color=ForestGreen] (0,3) -- (0,4.187);
		\draw [->, color=ForestGreen] (-1.5,2.598) -- (-1.075,3.626);
		\draw [->, color=ForestGreen] (-2.598,1.5) -- (-1.862,2.093);
		\draw [->, color=ForestGreen] (-3,0) -- (-2.150,0);
		\draw [->, color=ForestGreen] (-2.598,-1.5) -- (-1.862,-2.093);
		\draw [->, color=ForestGreen] (-1.5,-2.598) -- (-1.075,-3.626);
		\draw [->, color=ForestGreen] (0,-3) -- (0,-4.187);
		\draw [->, color=ForestGreen] (1.5,-2.598) -- (1.075,-3.626);
		\draw [->, color=ForestGreen] (2.598,-1.5) -- (1.862,-2.093);
		\draw (-1,1) node {$B_1$};
		\draw (-2.5,2.5) node {$B_3$};
		\draw[color=ForestGreen] (0,4) node[right] {$f_T$};
	\end{tikzpicture}
	\begin{tikzpicture}[line cap=round,line join=round,>=stealth,x=1cm,y=1cm, scale=0.75, line width=1.1pt]
		\clip(-6,6) rectangle (10.,12);
		\draw [->,color=red] (1.57,10) -- (1.57,11.187);
		\draw [->,color=red] (-1.57,10) -- (-1.57,11.187);
		\draw [->,color=cyan] (-1.57,8) -- (4.61,8.396);
		\draw [->,color=cyan] (1.57,8) -- (7.75,8.396);
		\draw [->,color=red] (-3.14,10) -- (-3.14,9.15);
		\draw [->,color=red] (0.524,10) -- (0.84,9.8);
		\draw [->,color=red] (1.047,10) -- (1.28,10.78);
		\draw [->,color=red] (-2.094,10) -- (-1.86,10.78);
		\draw [->,color=red] (-1.047,10) -- (-1.28,10.78);
		\draw [->,color=red] (-0.524,10) -- (-0.84,9.8);
		\draw [->,color=red] (0,10) -- (0,9.15);
		\draw [->,color=red] (3.14,10) -- (3.14,9.15);
		\draw [->,color=red] (2.618,10) -- (2.3,9.8);
		\draw [->,color=red] (2.094,10) -- (1.86,10.78);
		\draw [->,color=red] (-2.618,10) -- (-2.3,9.8);
		\draw [->,color=cyan] (0,8) -- (6.18,7.72);
		\draw [->,color=cyan] (0.524,8) -- (7.02,7.93);
		\draw [->,color=cyan] (1.047,8) -- (7.46,8.26);
		\draw [->,color=cyan] (-3.14,8) -- (3.14,7.72);
		\draw [->,color=cyan] (-2.618,8) -- (3.99,7.93);
		\draw [->,color=cyan] (-2.094,8) -- (4.32,8.26);
		\draw [->,color=cyan] (-1.047,8) -- (4.9,8.26);
		\draw [->,color=cyan] (-0.524,8) -- (5.34,7.93);
		\draw [->,color=cyan] (2.094,8) -- (8.04,8.26);
		\draw [->,color=cyan] (2.618,8) -- (8.48,7.93);
		\draw [->,color=cyan] (3.14,8) -- (9.42,7.72);
		\draw[->,line width=0.8pt] (-4,7)--(4,7) node[below] {$\theta$};
		\draw[line width=0.5pt] (4,8)--(-4,8) node[left] {$1$};
		\draw[line width=0.5pt] (4,10)--(-4,10) node[left] {$3$};
		\draw[->,line width=0.8pt] (-4,7)--(-4,11) node[left] {$r$};
		\draw[dotted,line width=0.5pt] (-3.14,10)--(-3.14,7) node[below] {$-\pi$};
		\draw[dotted,line width=0.5pt] (0,10)--(0,7) node[below] {$0$};
		\draw[dotted,line width=0.5pt] (3.14,10)--(3.14,7) node[below] {$\pi$};
		\draw[color=cyan] (5.5,9) node {$F_T$ on $\partial \Bs_1$};
		\draw[color=red] (5.5,10) node {$F_T$ on $\partial \Bs_3$};
	\end{tikzpicture}
	
	\begin{tikzpicture}[line cap=round,line join=round,>=stealth,x=1cm,y=1cm,scale=0.75,line width= 1.1pt,]
		\clip(-6,-4) rectangle (10.,5);
		\draw [line width=0.8pt] (0,0) circle (1cm);
		\draw [->,color=cyan] (1,0) -- (7.182,-0.28);
		\draw [->,color=cyan] (0,1) -- (6.182,1.396);
		\draw [->,color=cyan] (0.866,0.5) -- (6.73,0.43);
		\draw [->,color=cyan] (0.5,0.866) -- (6.448,1.126);
		\draw [->,color=cyan] (0.866,-0.5) -- (7.364,-0.57);
		\draw [->,color=cyan] (0.5,-0.866) -- (6.915,-0.606);
		\draw [->,color=cyan] (0,-1) -- (6.182,-0.604);
		\draw [->,color=cyan] (-0.866,0.5) -- (5.632,0.43);
		\draw [->,color=cyan] (-0.5,0.866) -- (5.915,1.126);
		\draw [->,color=cyan] (-0.866,-0.5) -- (4.998,-0.57);
		\draw [->,color=cyan] (-0.5,-0.866) -- (5.448,-0.606);
		\draw [->,color=cyan] (-1,0) -- (5.182,-0.28);
		\draw [line width=0.8pt] (0,0) circle (3cm);
		\draw [->,color=red] (3,0) -- (3,-0.85);
		\draw [->,color=red] (0,3) -- (0,4.187);
		\draw [->,color=red] (2.598,1.5) -- (2.28,1.3);
		\draw [->,color=red] (1.5,2.598) -- (1.266,3.378);
		\draw [->,color=red] (2.598,-1.5) -- (2.914,-1.7);
		\draw [->,color=red] (1.5,-2.598) -- (1.733,-1.818);
		\draw [->,color=red] (0,-3) -- (0,-1.813);
		\draw [->,color=red] (-1.5,2.598) -- (-1.267,3.378);
		\draw [->,color=red] (-2.598,-1.5) -- (-2.916076211353317,-1.7);
		\draw [->,color=red] (-3,0) -- (-3,-0.85);
		\draw [->,color=red] (-1.5,-2.598) -- (-1.734,-1.818);
		\draw [->,color=red] (-2.598,1.5) -- (-2.282,1.3);
		\draw (-1,1) node {$B_1$};
		\draw (-2.5,2.5) node {$B_3$};
		\draw[color=cyan] (5,2) node {$F_T\circ \Psi^{-1}$ on $\partial B_1$};
		\draw[color=red] (5,3) node {$F_T\circ \Psi^{-1}$ on $\partial B_3$};
	\end{tikzpicture}
	\captionsetup{singlelinecheck=off}
	\caption[.]{The maps $f_T$ on $\R^2$ (above), $F_T$ on $\Hs$ (middle) and $F_T\circ\Psi^{-1}$ on $\R^2$ (below) for the dynamics described in Example \ref{ex:planar}. In particular we notice that \vspace{-3mm}
		\begin{gather*}
			\deg(f_T,B_1,0)=\deg(f_T,B_3,0)=-1\\
			\deg(F_T\circ \Psi^{-1},B_1,0)=\mydeg(F_T,B_1,0)=0\\
			\deg(F_T\circ \Psi^{-1},B_3,0)=\mydeg(F_T,B_3,0)=-2 
		\end{gather*}
		\\[-8mm] The first equivalence in Theorem \ref{th:degree} is completed by  $\Sigma(F_T,\partial B_3)=\{0\}$, $\Sigma(F_T,\partial B_1)=\{1\}$ and $\mydeg(F_T,B_1,\nu_1)=-2$.
	}
	\label{fig:twist}
\end{figure}

We are now ready to present a first explicit example, illustrating how Theorem \ref{th:main} is successful in situations where, instead, classical Brouwer's degree is totally ineffective.
More precisely, we describe a family of situations where $\deg(f_T,\inn U,0)=\deg(f_T,\out U,0)=-1$ and the twist condition \eqref{eq:twist} holds, so that, regardless of having $\deg(f_T,A,0)=0$, by Theorem \ref{th:main} we obtain two periodic solutions starting in the annulus $A$. The dynamics is illustrated in Figure \ref{fig:twist}, for a specific choice of the parameters reported at the end of the example.

\begin{example} \label{ex:planar} 
	Let us fix $\tau\in(0,T)$.
		We consider the planar system $\dot z=h(t,z)$, where $h\colon\R\times \R^2\to\R^2$ is continuous for $t\in[0,T]\setminus\{\tau\}$, $T$-periodic in $t$ and continuously differentiable in $z$. Assume moreover that
	\begin{itemize}
		\item $h(t,z)=L_\infty z+p(t,z)$, where $\abs{p(t,z)}<C_p$ for every $(t,z)\in\R^3$, and
		$L_\infty$ is a symmetric matrix with $\det L_\infty=-1$;
		\item $h(t,0)=0$ and $D_z h(t,0)=L_0(t)$ for every $t\in\R$, where
		\begin{equation*}
			L_0(t)=\begin{cases}\frac{2\pi}{\tau}{\footnotesize\begin{pmatrix}  0&1\\-1&0	\end{pmatrix}} &\text{for $t\in[0,\tau)$}\\
				\B_0 &\text{for $t\in[\tau,T)$}
			\end{cases}
		\end{equation*} for some symmetric matrix $\B_0$ with $\det \B_0=-1$.
	\end{itemize}
	We show that the system has at least two $T$-periodic solutions in addition to the trivial one $z\equiv 0$.
	
	We define $\phi\colon[0,T]\times \R^2 \to \R^2$ as the evolution map associated to the solutions of the equations, noticing that we have global existence and uniqueness of solution, and set $f_T,F_T$ as in Section \ref{sec:flow}. Similarly, we denote with $\phi^{L_0}$ and $\phi^{L_\infty}$ respectively the evolution maps of the linear systems $\dot z=L_0(t)z$ and $\dot z=L_\infty z$, and $f_T^{L_0},F_T^{L_0},f_T^{L_\infty},F_T^{L_\infty}$ accordingly. Notice also that $\Null=\{0_{\R^2}\}$ for all three systems. We observe that $\phi^{L_0}(\tau,\cdot)=\I$, so that $\phi^{L_0}(T,z)=e^{(T-\tau)\B_0}z$. By Lemma \ref{lem:asymp_zero} we have that, for a sufficiently small $\inn r>0$:
	\begin{align*}
		&\deg (f_T,B_{\inn r},0)=\deg (f_T^{L_0},B_{\inn r},0)=-1 \\ &\Sigma(F_T,\partial B_{\inn r})=\Sigma(F_T^{L_0},\partial B_{\inn r})=\{1\}
	\end{align*}
Noticing that Lemma \ref{lem:asymp_inf} can be easily extended to allow the switching of the dynamics at $t=\tau$, we have that, for $\out r>0$ sufficiently large:
	\begin{align*}
		&\deg (f_T,B_{\out r},0)=\deg (f_T^{L_\infty},B_{\out r},0)=-1 \\ &\Sigma(F_T^{L_\infty},\partial B_{\out r})=\Sigma(F_T,\partial B_{\out r})=\{0\}\,.
	\end{align*}
	Applying both points in Theorem \ref{th:main} to the annulus $A=\clos{B_{\out r}}\setminus B_{\inn r}$ we obtain two  (nontrivial) $T$-periodic solutions: the first does not rotate around the origin, while the second makes exactly one rotation around the origin.
	
	In Figure \ref{fig:twist}, the maps $f_T$, $F_T$ and $F_T\circ\Psi^{-1}$ are illustrated on the boundary of $A$ for $\inn r=1$ and $\out r=3$  in the special case where
	\begin{equation*}
		\B_0=\begin{pmatrix}
			-1 &0\\ 0&1
		\end{pmatrix} \,, \qquad 	L_\infty=\begin{pmatrix}
			-1 &0\\ 0&1
		\end{pmatrix}\,,
	\end{equation*} $h(t,x)=L_0(t)x$ for $\abs{x}<1.5$,  $h(t,x)=L_\infty x$ for $\abs{x}>2$.
\end{example}

In Example \ref{ex:planar} we presented a very simple behaviour at zero and at infinity,  in order to make the degree and rotational properties very easy to recognize. However, it is possible to characterize such features in much more general situations. The most remarkable case is that of (nonresonant) linear planar Hamiltonian systems, where such information is contained in  Maslov's index (sometimes called also Conley--Zehnder index) associated to the system \cite{Abb}. Writing $J=\left(\begin{smallmatrix}
	0&-1\\1&0
\end{smallmatrix}\right)$, we recall the following properties (cf.~\cite[Lemma 4]{MaReZa}, bearing in mind that we are counting clockwise rotations).
\begin{lemma}\label{lem:degrot_linear}
	Let us consider the linear planar Hamiltonian system $\dot z=JL(t)z$, where $L(t)$ is symmetric, continuous in $t$, and the system is $T$-nonresonant. We denote with $\phi$ the evolution map associated to the system, and with $i_T$ its Maslov's index with respect to the interval $[0,T]$. Then, if $i_T=2k$ is even, we have, for every $r>0$
	\begin{align*}
		\deg(f_T,B_r,0)=-1\,,&& -k-\frac{1}{2}<\rot(F_T,\partial B_r)<-k+\frac{1}{2} \,.
	\end{align*}
	If, instead, $i_T=2k+1$ is odd, we have, for every $r>0$
	\begin{align*}
		\deg(f_T,B_r,0)=1\,,&& -k-1<\rot(F_T,\partial B_r)<-k
	\end{align*}
	(where $f_T$ and $F_T$ has the usual meaning defined in Section \ref{sec:flow}).
\end{lemma}

Let us now consider the planar Hamiltonian system 
\begin{equation}\label{eq:hamsist}
	\dot z= JD_z H(t, z)\,,
\end{equation}
 where $H\colon\R\times \R^2\to \R$ is continuous, $T$-periodic in $t$, and continuously
differentiable in $z$, with $D_zH(t,\cdot)$ Lipschitz continuous uniformly
in time.  As usual, when discussing the dynamics defined by \eqref{eq:hamsist}, we denote with $\phi\colon[0,T]\times \R^2 \to \R^2$ the evolution map associated to the solutions of the system,  and set $f_T,F_T$ accordingly, as in Section \ref{sec:flow}.
Combining the results above we obtain the following theorem, cf.~\cite[Theorem 13]{GidMar}.

\begin{theorem}\label{th:gidmar} Let us consider the Hamiltonian system \eqref{eq:hamsist}. Suppose that the Hamiltonian function $H(t, z)$ is twice differentiable at the origin $z = 0$ with respect to the space variable $z$, with $D_z H(t, 0) = 0$ and $D_{zz}H(t, 0) = \B_0(t)$, with $\B_0$ continuous, and that there exists a matrix $\B_\infty(t)$ continuous in time such that 
	\begin{equation*}
		\lim_{\abs{z}\to+\infty} \frac{\abs{D_zH(t,z)-\B_\infty(t)z}}{\abs{z}}=0 \qquad\text{uniformly in $t\in[0,T]$\,.}
	\end{equation*} 
Moreover, suppose that the linear systems at zero and infinity are
	$T$-nonresonant and denote respectively with $i_0$ and $i_\infty$ their $T$-Maslov indices. Then
	system \eqref{eq:hamsist} has at least $\abs{i_\infty-i_0}$ nontrivial $T$-periodic solutions (in addition to the trivial one $z\equiv 0$). Furthermore, if $i_0$ is even with $i_0\neq i_\infty$, then the number of nontrivial solutions is at least $\abs{i_\infty-i_0}+1$.
\end{theorem}

In \cite{GidMar}, Theorem \ref{th:gidmar} was proved using a simpler version of the degree $\mydeg$, whose values were computed manually as winding number.
We propose here a shorter proof of Theorem \ref{th:gidmar}, based on Theorem \ref{th:main}.  We refer to \cite{GidMar} for examples showing that the lower bound on the number of periodic solutions is sharp. 

First, we state a suitable variant of the Poincaré--Birkhoff Theorem for planar Hamiltonian systems. This statement is a special case of the main theorem in \cite{FonUre}. 
\begin{theorem}\label{th:PB}
Let us consider the Hamiltonian system \eqref{eq:hamsist}. Assume that there exist $\out r>\inn r>0$ and $n\in \Z$ such that 
\begin{equation*}
\rot(F_T,\partial B_{\inn r})<n<\rot(F_T,\partial B_{\out r}) \quad \text{or}\quad \rot(F_T,\partial B_{\inn r})>n>\rot(F_T,\partial B_{\out r})  \,.
\end{equation*}
Then \eqref{eq:hamsist} has at least two $T$-periodic solutions $\bar z_1, \bar z_2$, such that $\bar z_1(0), \bar z_2(0)\in A=B_{\out r}\setminus\clos{B_{\inn r}}$ and $\rot(F_T,\bar z_1(0))=\rot(F_T,\bar z_2(0))=n$.
\end{theorem}

\begin{proof}[Proof of Theorem \ref{th:gidmar}]
We denote with $r_0,r_\infty$ two radii sufficiently small, resp.~large, in the sense of Lemmata \ref{lem:asymp_zero} and \ref{lem:asymp_inf}, taking $\epsilon$ sufficiently small to guarantee the invariance of $\Sigma$. Assuming $i_0\neq i_\infty$ we notice that, by Lemma \ref{lem:degrot_linear}, \eqref{eq:twist} holds and:
\begin{itemize}
	\item for every $n\in \Z$ such that $i_0<2n<i_\infty$ (or $i_\infty<2n<i_0$) there exists two $T$-periodic solutions making exactly $-n$ rotations around the origin. Such solutions are obtained by Theorem \ref{th:PB}, noticing that $\rot (F_T,\partial B_{r_0})>-n>\rot (F_T,\partial B_{r_\infty})$ (resp.~$\rot (F_T,\partial B_{r_0})<-n<\rot (F_T,\partial B_{r_\infty})$). 
	\item if $i_\infty=2k$ is even, then $\deg(f_T,B_{r_\infty},0)=-1$ and $\Sigma(F_T,\partial B_{r_\infty})=\{-k\}$. Hence by Theorem \ref{th:main} there exists a $T$-periodic solution making exactly $-k$ rotations around the origin.
	\item if $i_0=2\ell$ is even, then $\deg(f_T,B_{r_0},0)=-1$ and $\Sigma(F_T,\partial B_{r_0})=\{-\ell\}$. Hence by Corollary \ref{corol:areapres} there exist two $T$-periodic solutions, each making exactly $-\ell$ rotations around the origin.
\end{itemize}
\end{proof}

Moreover, we remark that in Theorem \ref{th:gidmar} the Hamiltonian structure is necessary only in the part where the Poincaré--Birkhoff Theorem is used, so that we have the following result, valid for systems that may not be Hamiltonian.
\begin{prop}
	Let $H$ be as in Theorem \ref{th:gidmar}, with $i_0\neq i_\infty$, and let $p\colon\R\times \R^2\to \R$ be continuous, $T$-periodic in $t$ and Lipschitz continuous in $z$. Suppose that $p$ is differentiable at $z=0$ with respect to $z$, with $p(t,0)=D_zp(t,0)=0$ for every $t$ and $p(t,z)/\abs{z}\to 0$ uniformly in $t$ as $\abs{z}\to+\infty$. Then the system $\dot z=JD_zH(t,z)+p(t,z)$ has
	\begin{itemize}
		\item if $i_0$ is even, a $T$-periodic solution making $-i_0/2$ rotations around the origin;
		\item if $i_\infty$ is even, a $T$-periodic solution making $-i_\infty/2$ rotations around the origin.
	\end{itemize}
\end{prop}
\begin{proof}
	The result is a direct corollary of Theorem \ref{th:main}, combined with Lemmata \ref{lem:asymp_zero}, \ref{lem:asymp_inf} and \ref{lem:degrot_linear}.
\end{proof}

\subsection{A second example with zero Brouwer's degree}
We now present a basic example showing how our theory is effective in situation which cannot be  approached with classical methods.
		We consider the planar system $\dot z=h(t,z)$, where $h\colon\R\times \R^2\to\R^2$ is continuous for $t\in[0,T]\setminus\{T/2\}$, $T$-periodic in $t$ and continuously differentiable in $z$. We set, for some $r>0$,
	\begin{equation*}
		\text{for every $\abs{z}\geq r$,}\quad h(t,z)=\begin{cases}\frac{4\pi}{T}{\footnotesize\begin{pmatrix}  0&1\\-1&0	\end{pmatrix}} z &\text{if $t\in[0,T/2)$}\\
		v &\text{if $t\in[T/2,T)$}
		\end{cases}  
	\end{equation*} 
where $v$ is a vector with $\abs{v}=1$. With the usual notation from Section~\ref{sec:flow} and taking $R>r+T/2$ we have $f_T(x)=\frac{T}{2} v$ for every $x\in \partial U_R$ and thus $\deg(f_T,U_R,0)=0$. Hence, standard topological methods working on the Poincaré map cannot be applied to find a $T$-periodic solutions (including classical fixed-point theorems, which works in situation with non-zero degree), nor such an issue can be bypassed using continuation methods.

On the other hand, it is easily verified that $\Sigma(F_T,\partial U_R)=\{1\}$, which, by Theorem~\ref{th:degree}, implies $\mydeg(F_T,U_R,\nu_1)=-1$, hence the existence of a $T$-periodic solution of the system making exactly one rotation around the origin.

\subsection{Superlinearity at infinity} \label{sec:superlinear}
Our third example deals with another classical situation producing twist on an annulus: a superlinear behaviour at infinity produced by a superlinear second order ODE. In such a case classical topological methods can be applied to find a periodic solution, cf.~\cite{GidODE} and references therein; we briefly repropose here such result in our framework.
Multiplicity results have been obtained via the Poincaré--Birkhoff Theorem when an Hamiltonian structure is present. For multiplicity, a minor improvement can be obtained by our results, which we discuss with the purpose to show, once more, how our approach complements the  Poincaré--Birkhoff Theorem.

We study the second order ordinary differential equation
\begin{equation}\label{eq:Fode}
	\ddot{u}+g(t,u)+p(t,u,\dot u)=0 \,.
\end{equation}
Here $g\colon\R\times \R \to\R$ and $p\colon\R\times \R^2 \to\R$ are continuous, $T$-periodic in the first variable $t$ and continuously differentiable in the remaining variables. On the function $g$ we assume the superlinearity condition
\begin{equation}
	\lim_{\abs{u}\to\infty}\frac{g(t,u)}{u}=+\infty  \qquad \text{uniformly in $t\in[0,T]$}
\end{equation}
whereas $p$ satisfies the bound $\abs{p(t,u,w)}<C_p^0+C_p^1\abs{u}$ for two positive constants $C_p^0,C_p^1$. 
Notice that the scalar equation \eqref{eq:Fode} defines a planar dynamics on the phase plane, namely
\begin{equation} \label{eq:Fsystem}
	\begin{cases}
		\dot u =w\\
		\dot w = -g(t,u)-p(t,u,w)
	\end{cases}
\end{equation}
to which we associate the evolution map $\phi\colon [0,T]\times \Omega\to\R^2$, where the domain $\Omega\subseteq \R^2$ is yet to be discussed.

Intuitively, the key property of the dynamics is that all solutions starting outside a sufficiently large radius are expected to make an arbitrarily large number of rotations, thus producing twist on a sufficiently thick annulus. The main issue is that superlinearity does not guarantee global existence of solutions and thus the evolution map might not be defined on a suitable domain. For instance, in \cite{CofUll} an example is presented, for which Equation \eqref{eq:Fode} admits a solution exploding to infinity in finite time. Some additional assumptions might be required to overcome this problem. 
Recently, it has been shown in \cite{GidODE} that it is sufficient to require the continuability on $[0,T]$ for the solutions crossing the origin of the phase plane in that time interval. In particular, we have the following fact,  cf.~\cite[eq.~(4.19)]{GidODE}.

\begin{prop}\label{prop:gidU}
Let us assume that, for every $\bar t\in[0,T]$, the solutions of \eqref{eq:Fsystem} corresponding to the Cauchy problems $u(\bar t)=w(\bar t)=0$ are continuable on $[0,T]$.
Then there exists $\bar \alpha\in \R$ with the following property: for every $\alpha>\bar \alpha$  there exists an open bounded set $U_\alpha\subset \R^2$ such that for every $(t,z)\in[0,T]\times \clos {U_\alpha}$ the evolution map $\phi(t,z)$ associated to \eqref{eq:Fsystem} is well-defined, $\Null\subset U_\alpha$ and $\rot(F_T,\partial U_\alpha)=\{\alpha\}$.
\end{prop}

The existence of a $T$-periodic solution follows straightforwardly via the Poincaré--Bohl Theorem or, similarly, via Theorem \ref{th:deg1}  (cf.~\cite[Theorem~4]{GidODE}).

\begin{theorem}\label{th:GidODE}
Let us assume that, for every $\bar t\in[0,T]$, the solutions of \eqref{eq:Fode} corresponding to the Cauchy problem $u(\bar t)=\dot u(\bar t)=0$ are continuable on $[0,T]$. Then \eqref{eq:Fode} admits at least one $T$-periodic solution.
\end{theorem}
\begin{proof}
	By Proposition \ref{prop:gidU} and Theorem \ref{th:deg1}, taking $\alpha\notin \Z$, $\alpha>\bar \alpha$, we deduce that  $\deg(f_T,U_\alpha,0)=1$, hence $\phi_T$ has a fixed point $\bar z(0)$ in $U_\alpha$, which corresponds to a $T$-periodic solution $\bar z$ of the planar system. Moreover, if the fixed point is unique, and thus isolated, by the additivity of the degree we must have $\ind(\phi_T,\bar z(0))=1$.
\end{proof}

We point out that  the continuability requirement can be weakened, obtaining a necessary and sufficient condition for the existence of a $T$-periodic solution \cite[Corollary~5]{GidODE}.

Once a first periodic solution is found, we might look for additional ones. First of all, we notice that, up to a suitable $T$-periodic change of variable, the first periodic solution of $\eqref{eq:Fode}$ can be transformed into $x\equiv 0$, preserving the superlinearity property. Hence, for the remaining of this section, let us assume that 
\begin{equation}
	g(\cdot,0)+p(\cdot,0,0)=0 \,.
\end{equation}
We observe that, in such a case, the continuability assumption in Proposition~\ref{prop:gidU} is automatically satisfied. We define
\begin{equation}
	L(t)=\begin{pmatrix}
		0 & 1\\ -g_u(t,0)-p_u(t,0,0) & -p_w(t,0,0)
	\end{pmatrix}
\end{equation}
where $g_u,p_u,p_w$ are the partial derivatives of $g$ and $p$, so that the system $\dot z=L(t)z$ is the linearization of \eqref{eq:Fsystem} in the origin. We denote with $\phi^{L}$ the evolution map of the linear systems $\dot z=L(t)z$ and set $f_T^{L},F_T^{L}$ accordingly, as in Section \ref{sec:flow}. We make the following non-resonance assumption
\begin{equation}\label{eq:nonres}
f^{L}(T,z)=0 \qquad\text{if and only if} \qquad z=0
\end{equation}
namely, that $z\equiv 0$ is the unique $T$-periodic solution of $\dot z=L(t)z$, so that $\deg(f^{L}_T,B_1,0)$ is well-defined. 

We recall that, given a linear system $\dot z=L(t)z$, we have, for every $r>0$, $\deg(f^{L}_T,B_1,0)=\deg(f^{L}_T,B_r,0)$ and $\rot(F_T^{L}, \partial B_1)=\rot(F_T^{L}, \partial B_r)$. 

We are now ready to apply Theorem~\ref{th:main} to the superlinear case \eqref{eq:Fode}.

\begin{theorem} \label{th:super2}
	Let us assume \eqref{eq:nonres} and $\deg(f^{L}_T,B_1,0)\neq 1$. Then \eqref{eq:Fode} admits a nontrivial $T$-periodic solution $\tilde u$ (in addition to the trivial one $u\equiv 0$). Furthermore, such solution has the same rotational properties of the linearization around the origin: more precisely $\rot (F_T,(\tilde u(0),\dot{\tilde u}(0)))\in\Sigma(F_T^{L}, \partial B_1)$.
\end{theorem}
\begin{proof}
Let us take $U_\alpha$ as provided by Proposition~\ref{prop:gidU}, with $\alpha>\bar \alpha$ and $\alpha\notin \Z$. Since, for our dynamics, $\Null=\{0_{\R^2}\}$, there exists $\hat r>0$ such that $\clos{B_{\hat r}}\subset U_\alpha$.

By Lemma~\ref{lem:asymp_zero} there exists $r_0>0$ sufficiently small such that $r_0<\hat r$ and
\begin{gather*}
	\deg(f_T,B_{r_0},0)=\deg(f^{L}_T,B_{r_0},0)=\deg(f^{L}_T,B_1,0)\neq 1\\
	\Sigma(F_T, \partial B_{r_0})=\Sigma(F_T^{L}, \partial B_{r_0})=\Sigma(F_T^{L}, \partial B_1) \,.
\end{gather*}
On the other hand, we know by Theorem~\ref{th:deg1} that $\deg (F_T,U_\alpha,0)$. 
The result follows from Theorem~\ref{th:main} applied to the annulus $\clos{U_\alpha}\setminus B_{r_0}$.
\end{proof}

We now consider the case $p\equiv 0$, so that \eqref{eq:Fsystem} acquires a Hamiltonian structure and, given \eqref{eq:nonres}, we can define the $T$-Maslov's index $i_0$ associated with the linear system $\dot z=L(t)z$ (i.e.~with $\dot z=J\widetilde L(t)z$ for $\widetilde L(t)=-JL(t)$). Notice that, since solution can make more than half rotation only clockwise, the index $i_0$ will be non-positive. We denote with $\lfloor q\rfloor$ the integral part of a real number $q$.

\begin{corol}\label{cor:super_inf}
Let us assume $p\equiv 0$ and \eqref{eq:nonres}, denoting with $i_0$ the $T$-Maslov associated to  $\dot z=L(t)z$. We set $n_0:=\lfloor -i_0/2\rfloor$. Then
\begin{itemize}
	\item for every integer $K\geq n_0+1$, there exist two distinct $T$-periodic solutions of \eqref{eq:Fode}, each having exactly $2K$ zeros in $[0,T)$;
	\item if $i_0=2n_0$ is even,  there exists a $T$-periodic solution of \eqref{eq:Fode} having exactly $2n_0$ zeros in $[0,T)$.
\end{itemize}
\end{corol}
\begin{proof} First of all, we notice that, for a $T$-periodic solution of \eqref{eq:Fode}, having exactly $2K$ zeros in $[0,T)$ is equivalent to making exactly $K$ rotations around the origin. Indeed, since the origin cannot be crossed, the $\dot x$-axis of the phase-plane can be crossed only clockwise, hence exactly twice for every rotation.
	
	The first point is based on the Poincaré--Birkhoff Theorem, as was first proved by Jacobowitz \cite{Jacob} and Hartman \cite{Hart}. For a more recent, but analogous, approach, we suggest \cite{FonSfe}, where such results is generalized to systems. In particular, our statement follows from the one-dimensional case in \cite[Theorem~1]{FonSfe}, noticing that, by Lemmata \ref{lem:asymp_zero} and \ref{lem:degrot_linear}, we have $\rot(F_T,\partial u_\alpha)<n_0+1$, hence the condition $K\geq n_0+1$ is optimal in \cite[Lemma~3]{FonSfe}.
	
	The second part is a straightforward consequence of Theorem~\ref{th:super2}, with the rotational and degree properties of the linearization in the origin following from Lemmata \ref{lem:asymp_zero} and \ref{lem:degrot_linear}.
\end{proof}

Corollary \ref{cor:super_inf} illustrates once again the complementarity of the Poincaré--Birkhoff Theorem and our twist Theorem~\ref{th:main}, which we have already noticed in Theorem \ref{th:gidmar} for asymptotically linear systems. 

\medskip
\textbf{Acknowledgements.} The Author was partially supported by the GA\v{C}R Junior Star Grant 21-09732M, while he was a researcher at the Institute of Information Theory and Automation of the Czech Academy of Sciences. The Author is a member of the National Group for Mathematical Physics (GNFM--INdAM).


\end{document}